%% file: main.tex
\documentclass{article}
\usepackage[utf8]{inputenc}
\usepackage{amsmath,amssymb,amsthm}
\usepackage{hyperref,url}
\usepackage{graphicx}
\usepackage{float}
\usepackage{dsfont}
\usepackage{breqn}

\title{Equivalence of zero entropy and the Liouville property for stationary random graphs}

\usepackage{lipsum}

\newcounter{quotecount}
\newcommand{\MyQuote}[1]{\vspace{0.3cm}\refstepcounter{quotecount}%
     \parbox{10cm}{#1}\hspace*{2cm}(\arabic{quotecount})\\[0.3cm]}

\newcommand{\R}{\mathbb{R}}
\renewcommand{\H}{\mathbb{H}}
\newcommand{\Z}{\mathbb{Z}}
\newcommand{\N}{\mathbb{N}}
\newcommand{\E}{\mathbb{E}}
\renewcommand{\P}{\mathbb{P}}

\renewcommand{\d}{\mathrm{d}}

\newcommand{\g}{>}
\renewcommand{\l}{<}

\newcommand{\degree}{\text{deg}}

\newcommand{\G}{\mathcal{G}}

\newcommand{\F}{\mathcal{F}}
\renewcommand{\G}{\mathcal{G}}

\newcommand{\inv}{\mathrm{inv}}

\newtheorem{question}{Question}[section]

\newtheorem*{maintheorem}{Main Theorem}
\newtheorem{theorem}{Theorem}[section]
\newtheorem{lemma}[theorem]{Lemma}
\newtheorem{corollary}[theorem]{Corollary}

\newenvironment{remark}[1][Remark.]{\begin{trivlist}
\item[\hskip \labelsep \textbf{#1}]}{\end{trivlist}}

\author{Matias Carrasco Piaggio and Pablo Lessa}

\begin{document}

\maketitle

\begin{abstract}
We prove that any stationary random graph satisfying a growth condition and having positive entropy almost surely admits an infinite dimensional space of bounded harmonics functions.  Applications to random infinite planar triangulations and Delaunay graphs are given.
\end{abstract}
\begin{quote}
\end{quote}
\begin{quote}
\end{quote}

\section{Introduction}

A stationary random graph is a random rooted graph whose distribution is invariant under re-rooting by a simple random walk.  This notion was made explicit by Benjamini and Curien in \cite{benjamini-curien2012} motivated by several examples, including the Uniform Infinite Planar Triangulation/Quadrangulation (UIPT/Q), and previously defined notions such as unimodular random graphs.

In said work they develop the basic entropy theory for stationary random graphs, analogous to the well known theory for random walks on finitely generated groups, see \cite{kaimanovich-vershik1983}. In particular, they define an entropy and prove that if it is zero then the random graph almost surely satisfies the Liouville property (i.e. bounded harmonic functions are constant).  The converse implication, that positive entropy implies the existence of non-constant bounded harmonic functions, was posed as a question, see \cite[Remark 3.7]{benjamini-curien2012}.

In this work we answer this question in the afirmative under an additional condition on the stationary random graph.  The hypothesis is the following (see Lemma \ref{speedlemma}):

\MyQuote{\label{condition}The expectation of the number of elements of the ball of radius $n$ has finite exponential growth.}

\noindent Our main result is the following (see Theorem \ref{entropytheorem}).
\begin{maintheorem}\label{maintheorem}
An ergodic stationary random graph satisfying condition (\ref{condition}) above, has zero entropy if and only if it satisfies the Liouville property almost surely.  Furthermore, if such a graph has positive entropy, then almost surely it admits an infinite dimensional space of bounded harmonic functions.
\end{maintheorem}

Recent work of Benjamini, Paquette, and Pfeffer implies that the space of bounded harmonic functions on a stationary random graph must be either infinite or one dimensional (see \cite{benjamini-paquette-pfeffer2014}).  This yields an alternate proof of the second part of the above theorem (using the first part).

From the direct implication, which was already proved in their paper, Benjamini and Curien proved that the Uniform Infinite Planar Quadrangulation almost surely satisfies the Liouville property.  With the extension given by our result above, it is possible to deduce that certain stationary random graphs admit many bounded harmonic functions.  We will discuss in Section \ref{applications} a few such examples, like the $\kappa$-Markovian infinite planar triangulations, introduced recently by Curien in \cite{curien2014}, and the Hyperbolic Poisson-Delaunay graph.

An important difficulty in applying the above theorem is the lack of general criteria for establishing that a stationary random graph has at most exponential volume growth, even when the distribution of the degree of the root is known to be well behaved.  In most cases where the growth of a stationary graph is known (e.g. the Uniform Infinite Planar Quadrangulation, or the Hyperbolic Poisson-Delaunay graph) it seems to have been established by ad-hoc, and some times very intricate, arguments.  Thus, the authors consider the following question to be important:
\begin{question}
Given a stationary random graph such that the degree of the root is well behaved.  Under what conditions can one deduce that the graph has at most exponential volume growth?
\end{question}

To the best of the author's knowledge there is no widely applicable answer to the above question available in the literature.   We discuss two relevant partial results in Section \ref{canopysection}.   First, we give an example, due to Asaf Nachmias, of a stationary random graph with super-exponential growth such that the degree of the root has finite mean, and in fact is comparable to a Poisson variable.  The example also has the special property that the degree of the root determines the entire graph up to rooted isomorphism.   Second, we prove that for unimodular graphs whose root has finite expectation, if the number of elements at distance $n$ from the root is asymptotically independent from the degree of the root, then the graph has at most exponential growth (see Lemma \ref{growthlemma}).

Our proof of the Main Theorem involves Derriennic's zero-two law, a sharp criterion for equivalence of the tail and invariant events of a Markov chain (see Corollary \ref{zerotwocorollary}), and a ``looping'' argument which allows us to avoid parity problems (see Figure \ref{figuralazos}).  In order to show that our results are valid for graphs with unbounded degree, we improve the inequalities between the linear drift and entropy from \cite[Proposition 3.6]{benjamini-curien2012} with essentially the same proof (see Lemma \ref{speedlemma}).  To show that positive entropy implies that the space of bounded harmonic functions is infinite dimensional, we relate the dimension of this space to the mutual information between the first $m$ steps of a random walk and its tail behavior (see Lemma \ref{informationandharmoniclemma}).

This occupies the first few sections of the paper.  In section \ref{applications} we discuss examples and applications of the main theorem to several examples, many of which were already known.


\section{Tail and invariant events}

In this section we introduce the terminology and notation to be used in the rest of the article. Throughout this article we use the word ``graph'' as a synonym for connected locally finite (i.e. each vertex has finite degree) undirected graph.  If $X$ is a graph, we denote by $V(X)$ the set of vertices of $X$ and by $E(X)$ the set of edges. We allow multiple edges and loops.

Consider for any graph $X$, the path space $\Omega$ whose elements are sequences $\omega = (x_0,x_1,\ldots)$ of vertices with the property that $x_n$ is a neighbor of $x_{n+1}$ for all $n\geq 0$.  The space $\Omega$ when endowed with the topology of coordinate-wise convergence is a Polish space.  We define the one step transition probability $p(x,y)$ between two vertices $x,y \in V(X)$ by 
\[p(x,y) = \frac{\text{number of edges connecting }x\text{ to }y}{\degree(x)},\]
where edges connecting $x$ to $x$ are only counted once in the denominator. The $n$-th step transition probability $p^n(x,y)$ is defined by
\[p^n(x,y) = \sum\limits_{x_1,\ldots,x_{n-1}}p(x,x_1)p(x_1,x_2)\cdots p(x_{n-1},y).\]

\noindent For each $x \in V(X)$, the distribution of the simple random walk starting at $x$ is the unique Borel probability $\P_x$ on $\Omega$ which satisfies
\[\P_x(x_0 = x,x_1 = a_1,\ldots,x_n = a_n) = p(x,a_1)p(a_1,a_2)\cdots p(a_{n-1},a_n)\]
for all sequences $a_1,\ldots,a_n \in V(X)$. A simple random walk on $X$ is a $V(X)$ valued random process $x_n$, indexed on $n = 0,1,\ldots$, whose distribution is of the form
\[\sum\limits_{x \in V(X)} \mu(x)\P_x,\]
where the initial distribution of the walk $\mu$ is a probability on $X$.

For each $n$, let $\F^n$ be the $\sigma$-algebra on $\Omega$ generated by $x_n,x_{n+1},\ldots$.  The tail $\sigma$ algebra $\F^\infty$ is defined by
\[\F^\infty = \bigcap\limits_{n} \F^n,\]
while the invariant $\sigma$-algebra is defined by
\[\F^\inv = \lbrace A \in \F^\infty:\text{ if }\omega = (x_0,x_1,\ldots) \in A\text{ then }\omega' = (x_1,x_2,\ldots) \in A\rbrace.\]

\noindent Suppose $x_n$ is a simple random walk on $X$ whose distribution we denote by $\P$.  We say that the tail and invariant $\sigma$-algebras are equivalent with respect to $x_n$ if for each $A \in \F^\infty$, there exists $B \in \F^\inv$ such that $\P(A \triangle B) = 0$, where $\triangle$ denotes symmetric difference.

\begin{remark}
Consider a graph consisting of a single edge which joins two distinct vertices $x$ and $y$. This is a simple example where $\F^\infty$ and $\F^\inv$ are not equivalent. All invariant events are trivial.  However, the tail event of being at $x$ for all large enough even times is not invariant and has intermediate probability if the initial distribution gives positive mass to both vertices.
\end{remark}


\section{The zero-two law}

In this section we discuss the criterion for equivalence of tail and invariant events proved by Derriennic. For this purpose, define for each vertex $x$ in a graph $X$, the quantities
\[\alpha_n(X,x) = \sum\limits_{y \in V(X)}|p^{n+1}(x,y) - p^n(x,y)|,\]
and let 
\[\alpha_\infty(X,x) = \lim\limits_{n \to +\infty}\alpha_n(X,x).\]

\noindent We restate \cite[Théorème 3]{derriennic1976} in our context.
\begin{theorem}[Derriennic]\label{derriennicstheorem}
Let $X$ be a graph.  For each $x \in V(X)$, the limit $\alpha_\infty(X,x)$ exists and one has
\[\sup\limits_{x \in X}\alpha_\infty(X,x) = 0\text{ or }2.\]
Furthermore, the above supremum is $0$ if and only if $\F^\infty$ and $\F^\inv$ are equivalent for all simple random walks on $X$.
\end{theorem}

\noindent We will need the following consequence of Derrienic's result.

\begin{corollary}\label{zerotwocorollary}
If $X$ is a graph such that $p(x,x) \ge 1/2$ for all $x \in V(X)$, then $\F^\infty$ and $\F^\inv$ are equivalent for every simple random walk on $X$.
\end{corollary}
\begin{proof}
For each $x \in V(X)$, we calculate
\begin{align*}
\alpha_n(X,x) &= \sum\limits_{y}|p^{n+1}(x,y) - p^n(x,y)| = \sum\limits_{y}|\sum\limits_{z}p^{n-1}(x,z)(p^2(z,y) - p(z,y))|
\\ &\le \sum\limits_{z}p^{n-1}(x,z)\sum\limits_{y}|p^2(z,y) - p(z,y)| = \sum\limits_{z}p^{n-1}(x,z)\alpha_1(X,z).
\end{align*}

\noindent On the other hand one has $p^2(z,z) \ge 1/4$ and $p(z,z) \ge 1/2$, so in particular
\[\alpha_1(X,z) \le 2-1/4,\]
for all $z \in Z$.

\medskip
\noindent This implies $\alpha_\infty(X,x) \le 2 - 1/4$ for all $x \in V(X)$, so by Theorem \ref{derriennicstheorem}, the tail and invariant $\sigma$-algebras are equivalent for all simple random walks on $X$, as claimed.
\end{proof}

\section{Mutual information}

The mutual information between two random variables is a non-negative (possibly infinite) number which quantifies the dependence relationship between them. In particular, the mutual information is zero if and only if the variables are independent, and is maximized when both variables coincide.  

In this section we consider the mutual information between the first $m$ steps of a simple random walk and all steps after time $n$, as well as the mutual information between the first $m$ steps and the tail behavior of the simple random walk on a graph $X$.  We review the basic properties relating these quantities to the space of bounded harmonic functions on the graph (see in particular \cite{blackwell1955},\cite{derriennic1985}, and \cite{kaimanovich1992}).  This will be useful later in our study of entropy of stationary random graphs.

Fix a graph $X$, a root vertex $x \in V(X)$, and recall that $\Omega$ denotes the space of paths $(x_0,x_1,\ldots)$ in $X$.   Denote by $\F_n$ the $\sigma$-algebra generated by $(x_0,\ldots,x_n)$ for each $n$.

Let $\widehat{\P}_x$ be the distribution of two identical copies of a simple random walk starting at $x$ in $X$, while $\P_x \times \P_x$ denotes the distribution of two independent random walks starting at $x$. Note that both probabilities are defined on $\Omega \times \Omega$ but the former is supported on the diagonal, while the later is not (save trivial examples).

Let $\varphi$ be the convex function given by $\varphi(t) = t\log(t)$.  For $m \l n \le \infty$, the mutual information between $\F_m$ and $\F^n$ is defined by
\[I_m^n(X,x) = \sup\left\lbrace \sum\limits_{i}\varphi\left(\frac{\widehat{\P}_x(A_i)}{(\P_x \times \P_x)(A_i)}\right)(\P_x \times \P_x)(A_i) \right\rbrace,\]
where the supremum is over all finite partitions of $\Omega\times \Omega$ whose sets $A_i$ belong to $\sigma(\F_m \times \F^n)$. It follows from the convexitiy of $\varphi$, that $I_m^n(X,x)$ is always defined and non-negative, and equals zero if and only if $\widehat{\P}_x$ and $\P_x\times \P_x$ coincide on $\sigma(\F_m \times \F^n)$.

Recall that a function $f:V(X) \to \R$ is said to be harmonic if 
\[f(y) = \sum\limits_{z \in V(X)} p(y,z)f(z)\]
for all $y \in V(X)$.  A graph is said to satisfy the Liouville property if and only if all its bounded harmonic functions are constant.  The following result shows that, under mild hypothesis, the mutual information $I_m^\infty(X,x)$ is directly related to the dimension of the space of bounded harmonic functions on the graph $X$.

\begin{lemma}\label{informationandharmoniclemma}
Let $(X,x)$ be a rooted graph such that $\F^\inv$ and $\F^\infty$ are equivalent for the simple random walk starting at $x$.  Then $X$ satisfies the Liouville property if and only if $I_m^\infty(X,x) = 0$ for all $m$.  Furthermore, if the space of bounded harmonic functions on $X$ is finite dimensional and of dimension $d$, then $I_m^\infty(X,x) \le \log(d)$ for all $m$.
\end{lemma}
\begin{proof}
By \cite[Theorem 2]{blackwell1955}, the bounded harmonic functions on $X$ are in bijection with bounded shift invariant measurable functions on the space of paths $\Omega$ considered modulo modifications on $\P_x$-null sets.  Since $\F^\inv$ and $\F^\infty$ are equivalent, this implies that $X$ satisfies the Liouville property if and only if $\F^\infty$ is trivial.

\medskip
\noindent If $\F^\infty$ is trivial, then $\F_m$ is independent from $\F^\infty$ for each $m$, so $I_m^\infty(X,x) = 0$ as claimed.  In the other direction, if $I_m^\infty(X,x) = 0$ for all $m$, then $\F_m$ and $\F^\infty$ are independent. Since one can approximate any tail event by events in $\F_m$ (for $m$ large), we obtain that each tail event is independent from itself. This implies that $\F^\infty$ is trivial as claimed.

\medskip
\noindent Suppose now that the space of bounded harmonic functions on $X$ has dimension $d$.  By Blackwell's result above, there is a partition $B_1,\ldots,B_d$ of $\Omega$ into disjoint tail events which are atoms in $\F^\infty$.  By Dobrushin's Theorem (see \cite[Lemma 7.3]{gray2011}), one may calculate $I_m^\infty$ as the supremum over all partitions of $\sigma(\F_m\times \F^\infty)$ of the form $A_i \times B_j$, where $A_1,\ldots, A_n \in \F_m$.  For any such partition, one has
\begin{align*}
\sum\limits_{i,j}\varphi\left(\frac{\P_x(A_i\cap B_j)}{\P_x(A_i)\P_x(B_j)}\right)&\P_x(A_i)\P_x(B_j)
\\ &= -\sum\limits_{j = 1}^d\P_x(B_j)\log\left(B_j\right) + \sum\limits_{i,j}\varphi\left(\frac{\P_x(A_i\cap B_j)}{\P_x(A_i)}\right)\P_x(A_i)
\\ &\le \log(d) + \sum\limits_{i = 1}^n \log\left(\sum\limits_{j=1}^d \frac{\P_x(A_i\cap B_j)^2}{\P_x(A_i)^2}\right)\P_x(A_i)
\\ &\le \log(d),
\end{align*}
where we use Jensen's inequality and the fact that $\sum_j p_j^2 \l 1$ if $\sum_j p_j = 1$ (in our case $p_j = \P_x(A_i \cap B_j)/\P_x(A_i)$).

\medskip
\noindent By taking supremum, one obtains $I_m^\infty(X,x) \le \log(d)$ as claimed.
\end{proof}

The following Lemma gives a concrete formula for the mutual information $I_m^n(X,x)$ in terms of the transition probabilities of the random walk.  It will be used later on when we consider the asymptotic entropy of random walks on stationary random graphs.

\begin{lemma}\label{informationformulalemma}
Let $(X,x)$ be a rooted graph. Then the following holds:
\begin{enumerate}
 \item For each finite $m \l n$, one has
 $$I_m^n(X,x) = \sum\limits_{y,z \in V(X)}\log\left(\frac{p^{n-m}(y,z)}{p^n(x,z)}\right)p^m(x,y)p^{n-m}(y,z).$$
 \item For each $m$, the function $n \mapsto I_m^n(X,x)$ is non-increasing and converges to $I_m^\infty(X,x)$ when $n \to +\infty$.
\end{enumerate} 
\end{lemma}
\begin{proof}
By Dobrushin's Theorem (see \cite[Lemma 7.3]{gray2011}), one may take the supremum in the definition of $I_m^n(X,x)$ over partitions in a generating set of $\sigma(\F_m\times \F^n)$.  The subsets of $\Omega\times \Omega$ consisting of pairs of paths $((x_i),(y_i))$ satisfying $x_0=a_1,\ldots,x_m=a_m, y_n=a_n,\ldots,y_N = a_N$ for fixed $a_i$ and $N \g n$, generate the necessary $\sigma$-algebra, and hence, we may take the supremum over partitions into sets of this form.

\medskip
\noindent For any fixed $N \g n>m$, consider the partition $\lbrace A_{j}\rbrace$ into sets as above, where $a_1,\ldots,a_m$, $a_n,\ldots,a_N$ range over all of $V(X)$. Because of the Markov property one obtains the same result for all $N$ in the following calculation
\begin{align*}
\sum\limits_{j}&\varphi\left(\frac{\widehat{\P}_x(A_j)}{(\P_x \times \P_x)(A_j)}\right)(\P_x \times \P_x)(A_j)
=  \sum_{a_n,a_m} \log\left(\frac{p^{n-m}(a_m,a_n)}{p^n(x,a_n)}\right)p^m(x,a_m)p^{n-m}(a_m,a_n).
\end{align*}
This implies the first claim by taking supremum. Since $I_m^n(X,x)$ is calculated as a supremum over a set of partitions which decreases with $n$, $n \mapsto I_m^n(X,x)$ is non-increasing, and the limit
\[L = \lim\limits_{n \to +\infty}I_m^n(X,x)\]
exists. It is no smaller than $I_m^\infty(X,x)$.  Notice that the formula for $I_m^n(X,x)$ implies that $I_m^{m+1}(X,x)$ is finite, and hence so is $L$.

\medskip
\noindent To simplify notation, fix $m$ and set $\G_m = \sigma(\F_m \times \F^n)$ for each $m \l n \le \infty$.  By the Gelfand-Yaglom-Perez Theorem (\cite[Theorem 2.4.2]{pinsker1964} or \cite[Lemma 7.4]{gray2011}), one has
\[I_m^n(X,x) = \int\limits_{\Omega \times \Omega} \varphi(f_n)\d(\P_x \times \P_x),\]
for all $n$ (including $n = \infty$) where $f_n$ is the Radon-Nikodym derivative of $\widehat{\P}_x$ restricted to $\G_n$ relative to $\P_x \times \P_x$ restricted to the same $\sigma$-algebra.

\medskip
\noindent By the reverse martingale convergence theorem (see \cite[pg. 483]{doob2001}), one has that $f_n \to f_{\infty}$ pointwise when $n \to +\infty$.  Hence, $\varphi(f_n)$ converges to $\varphi(f_{\infty})$ and it suffices to show that these functions are uniformly integrable to establish that $\lim I_m^n(X,x) = I_m^\infty(X,x)$.

\medskip
\noindent Since $f_n = \E(f_m | \G_n)$, and the conditional expectation is relative to $\P_x\times \P_x$, one obtains by Jensen's inequality that
\[e^{-1} \le \varphi(f_n) =  \varphi(\E(f_m|\G_n)) \le \E(\varphi(f_m)|\G_n)),\]
for all finite $n$. By the reverse martingale convergence theorem, the right-hand side converges in $L^1$ to $\E(\varphi(f_m)|\G_\infty))$, and therefore, the family $\varphi(f_n)$ is uniformly integrable as claimed.  It follows that $\lim\limits_{n \to +\infty}I_m^n(X,x) = I_m^\infty(X,x)$ which concludes the proof.
\end{proof}

\section{Linear drift and entropy of random graphs}

Consider the topological space whose points represent all isomorphism classes of rooted graphs.  A sequence of rooted graphs in this space converges if and only if the isomorphism type of the ball of each fixed radius around the root is eventually constant.  The resulting space is separable and its topology comes from a complete metric.

Furthermore, one can construct a larger space consisting of rooted graphs with a path starting at the root.  Given a random graph $(X,x)$, one can find a random element of the space of graphs with paths $(X,x,(x_0,x_1,\ldots))$ such that the conditional distribution of $(x_0,x_1,\ldots)$ given $(X,x)$ is that of a simple random walk on $(X,x)$ starting at $x$.  We call $(X,x,(x_0,x_1,\ldots))$ a simple random walk on $(X,x)$. Sometimes we omit $(X,x)$ and just write $x_n$.

A random graph $(X,x)$ is called stationary if it has the same distribution as $(X,x_1)$ where $x_n$ is a simple random walk on $(X,x)$.  A stationary random graph is ergodic if the distribution of the simple random walk on it is an ergodic invariant measure for the shift transformation
$$(X,x,(x_0,x_1,x_2,\ldots)) \mapsto (X,x_1,(x_1,x_2,\ldots)).$$

\noindent Let $x_n$ be the simple random walk on an ergodic stationary random graph $(X,x)$.  By Kingman's subadditive ergodic theorem, the limit
\[\ell(X,x) = \lim\limits_{n \to +\infty}\frac{d(x_0,x_n)}{n}\]
exists almost surely and in mean. Here $d(x_0,x_n)$ denotes the graph distance on the graph $(X,x)$ between $x_0$ and $x_n$.  We call $\ell(X,x)$ the linear drift of the simple random walk on $(X,x)$.  One obtains trivially that $0 \le \ell(X,x) \le 1$.

Another important quantity associated to the random walk $x_n$ is its entropy.  It is defined as the limit
\[h(X,x) = \lim\limits_{n \to +\infty}-\log(p^n(x_0,x_n))\]
which exists almost surely and in $L^1$ (again by Kingman's theorem) under the condition 
$$-\E(\log(p(x_0,x_1)) \le \E(\log(\degree(x))) \l +\infty.$$

\noindent Under a mild assumption on the growth of the random graph one can conclude that $h(X,x) = 0$ if and only if $\ell(X,x) = 0$.  The following proof is almost the same as that of \cite[Proposition 3.6]{benjamini-curien2012}, which itself follows closely preceding results, see the references preceding \cite[Theorem 13.31]{lyons-peres2014}. In the following statement, $|B_r(x)|$ denotes the number of elements in the set of vertices at distance $r$ or less from $x$.

\begin{lemma}\label{speedlemma}
Let $(X,x)$ be an ergodic stationary random graph satisfying the following assumption
\begin{equation}\label{assumption1}
v(X,x) = \liminf\limits_{r \to +\infty}r^{-1}\E\left(\log\left|B_r(x)\right|\right) \l +\infty.
\end{equation}
Then $h(X,x)$ is finite and satisfies the following inequalities
\[\frac{1}{2}\ell(X,x)^2 \le h(X,x) \le \ell(X,x)v(X,x).\]
\end{lemma}
\begin{proof}
By the Carne-Varopoulos inequalities, one obtains
\[p^r(x_0,x_r) \le 2\left(\frac{\degree(x_r)}{\degree(x_0)}\right)^{\frac{1}{2}}e^{-\frac{d(x_0,x_r)^2}{2r}},\ \text{ for all }r.\]
The lower bound follows by taking $-\log$ and limit. Notice that $\E(\log \degree(x))<+\infty$ by assumption (\ref{assumption1}), and therefore $\degree(x_r)/r \to 0$ almost surely by Birkhoff's theorem.

For the upper bound, we use the observation that the function 
$$(p_1,\ldots,p_r) \mapsto \sum p_i\log(1/p_i)$$ 
is maximized over all $p_1,\ldots, p_r \ge 0$, with $p_1+\cdots+p_r = 1$, when all the $p_i$ are equal to $1/r$.  This yields (see also the proof of \cite[Proposition 3.6]{benjamini-curien2012}) that
\[-\sum\limits_{y} p^r(x,y)\log(p^r(x,y)) \le (1-p_r)\log\left|B_{(\ell+\epsilon)r}(x)\right| + p_r\log\left|B_r(x)\right|,\]
for all $\epsilon \g 0$. Here $\ell = \ell(X,x)$, and $p_r = \P(d(x_0,x_r) \ge (\ell + \epsilon)r)$ goes to $0$ when $r \to +\infty$.  Dividing by $r$, taking expectation, inferior limit (at this point we use assumption (\ref{assumption1})), and then letting $\epsilon$ go to zero, yields the desired upper bound for $h(X,x)$.
\end{proof}

\section{Bounded harmonic functions and entropy of random graphs}

It was shown in \cite[Theorem 3.2]{benjamini-curien2012} that $h(X,x) = 0$ if and only if almost surely $\F^\infty$ is trivial for the simple random walk starting at the root of $(X,x)$.  It follows that if $h(X,x) = 0$, then $X$ almost surely satisfies the Liouville property.  The question of whether the converse always holds was posed in the same paper, see \cite[Remark 3.7]{benjamini-curien2012}.  We will settle this question under mild additional hypothesis.

We will show that if $(X,x)$ is an ergodic stationary random graph with positive entropy, then almost surely the space of bounded harmonic functions on $X$ is infinite dimensional.  In particular, the graph obtained by taking the disjoint union of two copies of Cayley graph of $\Z^3$ and adding an edge joining them cannot occur with positive probability for any stationary random graph since its space of bounded harmonic functions has dimension 2.  Also, any graph with transitive isomorphism group must either satisfy the Liouville property or have an infinite dimensional space of bounded harmonic functions.

To begin we express the entropy $h(X,x)$ of a stationary random graph as the average mutual information between the first step and the tail of the corresponding simple random walk.
\begin{lemma}\label{entropyformula}
Let $(X,x)$ be an ergodic stationary random graph with finite entropy.  Then for each $m$, one has
\[h(X,x) = \E\left(\frac{1}{m}I_m^\infty(X,x)\right).\]
\end{lemma}
\begin{proof}
By Lemma \ref{informationformulalemma}, $I_m^n(X,x)$ is non-increasing and converges to $I_m^\infty(X,x)$ when $n \to +\infty$. Hence,
\[\E(I_m^\infty(X,x)) = \lim\limits_{n \to +\infty}\E(I_m^n(X,x)).\]

\noindent Using the formula from Lemma \ref{informationformulalemma} and stationarity, one obtains
\begin{align*}
\E(I_m^n(X,x)) &= \E\left(\sum\limits_{y,z \in X}\log\left(\frac{p^{n-m}(y,z)}{p^n(x,z)}\right)p^m(x,y)p^{n-m}(y,z)\right) 
\\ &= \E\left(\log\left(\frac{p^{n-m}(x_m,x_n)}{p^n(x_0,x_n)}\right)\right)
\\ &= -\E\left(\log\left(p^n(x_0,x_n)\right)\right) + \E\left(\log\left(p^{n-m}(x_m,x_n)\right)\right)
\\ &= -\E\left(\log\left(p^n(x_0,x_n)\right)\right) + \E\left(\log\left(p^{n-m}(x_0,x_{n-m})\right)\right).
\end{align*}

\noindent Letting $H_n = -\E(\log(p^n(x_0,x_n))$, one has obtained that $H_n - H_{n-m}$ converges monotonely.  Since $\frac{1}{n}H_n$ converges to $h(X,x)$, the limit must be $mh(X,x)$ (take the telescoping sum over $n = km$ for $k\in \mathbb{N}$).  This concludes the proof.
\end{proof}

\noindent We can now prove our main theorem.
\begin{theorem}\label{entropytheorem}
Let $(X,x)$ be an ergodic stationary random graph satisfying the assumption (\ref{assumption1}) of Lemma \ref{speedlemma}.  Then $h(X,x) = 0$ if and only if almost surely $X$ satisfies the Liouville property.  Furthermore, if $h(X,x) \g 0$, then almost surely the space of bounded harmonic functions on $X$ is infinite dimensional.
\end{theorem}
\begin{proof}
By Lemma \ref{entropyformula}, for each $m$, one has
\[h(X,x) = \E\left(\frac{1}{m}I_m^\infty(X,x)\right).\]

\noindent Hence, if $h(X,x) = 0$, then almost surely one has $I_m^\infty(X,x) = 0$ for all $n$.  By Lemma \ref{informationandharmoniclemma}, this implies that $(X,x)$ satisfies the Liouville property almost surely as claimed.

\medskip
\noindent Assume now that $h(X,x) \g 0$. Notice that by Lemma \ref{speedlemma}, the linear drift of the random walk on $(X,x)$ is positive.  We consider a stationary random graph $(X',x)$ obtained from $(X,x)$ by adding $\degree(y)$ edges connecting each vertex $y$ to itself (see Figure \ref{figuralazos}).  The simple random walk on this new random graph differs from the old one by a geometric waiting time with expectation $2$.  In particular, the linear drift of the simple random walk on $(X',x)$ is also positive.  By Lemma \ref{speedlemma}, this implies $h(X',x) \g 0$.

\medskip
\noindent Using Lemma \ref{entropyformula} as above, one obtains
\[mh(X',x) = \E(I_m^\infty(X',x)),\]
so that 
\[\P\left(I_m^\infty(X',x) \ge mh(X',x)\right) \g 0.\]

\noindent Notice that $(X',x)$ almost surely satisfies the hypothesis of Lemma \ref{zerotwocorollary}, so that $\F^\inv$ and $\F^\infty$ are equivalent.  Therefore, we may apply Lemma \ref{informationandharmoniclemma} and obtain from the inequality above (choosing $m$ so that $mh(X',x) \g \log(d)$) that for each $d$, the probability that $(X',x)$ admits more than $d$ linearly independent bounded harmonic functions is positive.  Since the space of bounded harmonic functions remains unchanged by adding $\deg(y)$ loops at each vertex, one obtains that for each $d$, the probability that $(X,x)$ admits at least $d$ linearly independent bounded harmonic functions is positive.  Because $(X,x)$ is ergodic, almost surely the space of bounded harmonic functions on $(X,x)$ is infinite dimensional. This concludes the proof.
\end{proof}

\begin{figure}[H]
\centering
\begin{minipage}{0.4\textwidth}
\vspace{1.15cm}
\includegraphics{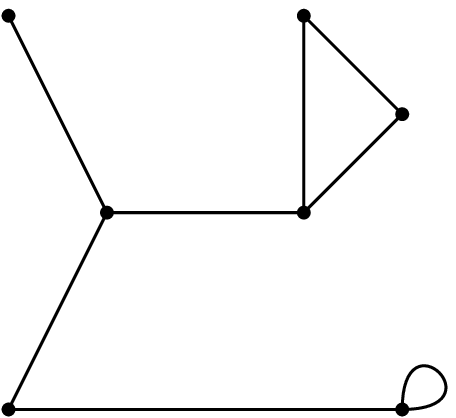}
\end{minipage}
\begin{minipage}{0.4\textwidth}
\includegraphics{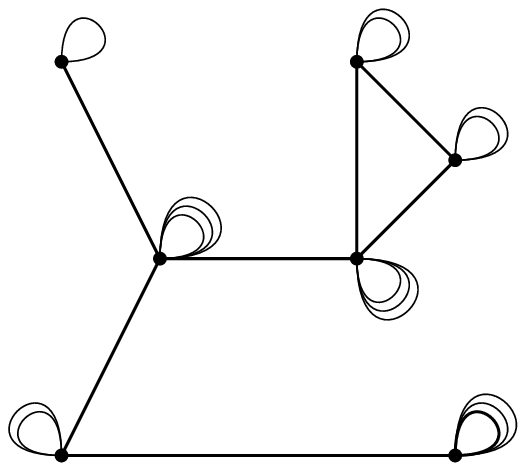}
\end{minipage}
\caption{\label{figuralazos}The stationary distribution $m(x) = \degree(x)/\sum \deg(y)$ of a finite graph does not change when one adds $\degree(x)$ loops at each vertex $x$.  
For general stationary random graphs this process does not affect stationarity.}
\end{figure}

\section{Applications\label{applications}}

\subsection{Random bridges on a tree}

As an application of Theorem \ref{entropytheorem} we will construct a stationary random graph whose random walk has positive linear drift.  It follows that almost surely the graph admits an infinite dimensional space of bounded harmonic functions, a property which, to the best of our knowledge, is not easily established by other means.

In the context of random walks on groups, a similar example, is given by the Cayley graph of the lamplighter group on $\Z^3$.  It is quite simple to establish that the random walk on this group has positive linear drift but, a priori, non-constant harmonic functions are not so easy to exhibit.  However, all harmonic functions have recently been explicitely identified thanks to the works of Erschler, Lyons and Peres (see \cite{erschler2011} and \cite{lyons-peres2014b}).

Our random graph is a simplified variant of the Stochastic Hyperbolic Infinite Quadrangulation, see \cite[Section 6.3]{benjamini2011}.  The difference is that we label edges with only $+1$ and $-1$ (never $0$) and we join the vertex at a corner to the first corner such that the sum along edges is $0$ (instead of $-1$ as in the SHIQ, see Figure \ref{figuraejemplo}).

This last modification implies that our random graph is regular, while the SHIQ almost surely has vertices with arbitrarily large degree.   However, our graph is non-planar, and the new edges connect points which are arbitrarly far away on the tree. In particular, the graph is not quasi-isometric to the tree.  Hence, even though our graph is transient, having the regular tree as a subgraph \cite{lyons1983}, the existence of non-constant bounded harmonic functions does not follow from \cite{benjamini-schramm1996}. In principle the graph might be almost planar, i.e. admit a quasi-monomorphism onto a planar graph, we conjecture that this is 
not the case.  We now give the details of the construction.

To begin, take a regular degree three tree $T_0$ with some fixed root $x$.  We consider this graph embedded in the plane without self crossings so that there is an order (say clockwise) among the three edges sharing each vertex.  We define a corner as the angular sector between two consecutive edges.  There is a partial order on the set of corners which is given by the clockwise contour of the graph.

A graph $(T,x)$ rooted at $x$ is constructed as follows:  A random label $+1$ or $-1$ is chosen with probability $1/2$ independently for each edge of $T_0$.  For each vertex $y$ and each corner at $y$, we add an edge joining $y$ to the vertex $z$ of the first corner in the partial order such that the sum of labels along the shortest path from $y$ to $z$ is equal to $0$.  It follows that the random graph $(T,x)$ is almost surely regular with all vertices of degree 9. In particular, the assumption (\ref{assumption1}) of Lemma \ref{speedlemma} is trivially satisfied.

We first show that $(T,x)$ is stationary. We do so by showing that it is unimodular, which is equivalent to stationarity since $(T,x)$ is regular (see below). 

Recall that a random rooted graph $(X,x)$ is said to be unimodular if for every function $F$, going from the space of isomorphism classes of graphs with two ordered roots to $[0,+\infty)$, one has
\[\E\left(\sum\limits_{y \in V(X)}F(X,x,y)\right) = \E\left(\sum\limits_{y \in V(X)}F(X,y,x)\right).\]
If a random rooted graph $(X,x)$ defined on some probability space $(\Omega,\F,\P)$ is unimodular and $\E(\deg(x)) \l +\infty$, then $X$ is stationary with respect to the probability measure $\mathbb{Q}$ defined by 
\begin{equation}\label{bias}
\frac{\d \mathbb{Q}}{\d \P}(X,x) = \frac{\deg(x)}{\E_\P(\deg(x))}. 
\end{equation}
See for example \cite[Section 2.2]{benjamini-curien2012}. Since $\P$ and $\mathbb{Q}$ are absolutely continuous, the almost sure properties of $(X,x)$ coincide with that of a stationary random graph.

\begin{lemma}
The random graph $(T,x)$ just constructed is stationary.
\end{lemma}
\begin{proof}
Suppose $L$ is a random labeling of the sides of the ternary tree. Given a vertex $y$ in the tree $T_0$, let $T(L,x,y)$ denote the isometry class, in the space of graphs with two ordered roots, of the graph $T$ obtained from the labeling $L$ with two ordered roots at $x$ and $y$ respectively.

The claim is that for every function $F$, going from the space of isomorphism classes of graphs with two ordered roots to $[0,+\infty)$, one has
\[\E\left(\sum\limits_{y\in V(T_0)}F(T(L,x,y))\right) = \E\left(\sum\limits_{y\in V(T_0)}F(T(L,y,x))\right).\]

Notice that since the vertices of $T$ are deterministic, it suffices to show that for each fixed $y$ one has
\begin{equation}\label{bridgesunimodular}\E\left(F(T(L,x,y))\right) = \E\left(F(T(L,y,x))\right).\end{equation}

To see this, we assume that the underlying ternary tree $T_0$ has been embedded into the Hyperbolic plane in such a way that all edges have the same length and meet at each vertex forming $120^\circ$ angles.

Under this assumption the hyperbolic central symmetry with respect to the midpoint of any edge leaves the graph invariant and hence uniquely determines an isomorphism of the tree.  Any such symmetry acts on a labeling $L$ in the obvious way.

Assume now that $y$ is a neighbor of $x$ in the ternary tree and let $\sigma$ be the hyperbolic central symmetry with respect to the midpoint of the edge joining $x$ to $y$.  Notice that the graph $T(L,y,x)$ is isomorphic to $T(\sigma_*L,x,y)$, where $\sigma_*L$ is the labeling $L$ rotated using $\sigma$. Since $\sigma_*L$ has the same distribution as $L$, this establishes (\ref{bridgesunimodular}) as claimed.

The general case follows by changing the labeling using the composition $\sigma_1\circ \cdots \circ \sigma_n$, where the $\sigma_i$ are the central symmetries with respect to the midpoints of the edges in the shortest path joining $x$ to $y$.
\end{proof}

We now verify that the simple random walk on $(T,x)$ has positive linear drift.

\begin{lemma}
The simple random walk on $(T,x)$ has positive linear drift.
\end{lemma}
\begin{proof}
Since $T$ is regular with degree $9$, has the same vertex set as the tree $T_0$, and contains $T_0$ as a subgraph, then $X$ satisfies the strong isoperimetric inequality with a deterministic isoperimetric constant.  Hence by \cite[Theorem 1.1]{virag2000} (see also \cite{gerl1988}) there exists a constant $\epsilon>0$ such that
$$\liminf_{n\to +\infty}\frac{d(x,x_n)}{n}\geq \epsilon$$
almost surely.
\end{proof}

Notice that the distribution of $(T,x)$ has compact support. Hence it can be written as the average of ergodic distributions by Choquet's theorem.  For almost all of these ergodic distributions, it follows from the previous lemma that the linear drift of the random walk is positive. Hence, by Theorem \ref{entropytheorem}, almost every graph admits an infinite dimensional space of bounded harmonic function.  We conclude that the dimension of the space of bounded harmonic functions on $(T,x)$ is infinite dimensional almost surely, as claimed.

\begin{figure}[H]
\centering
\includegraphics{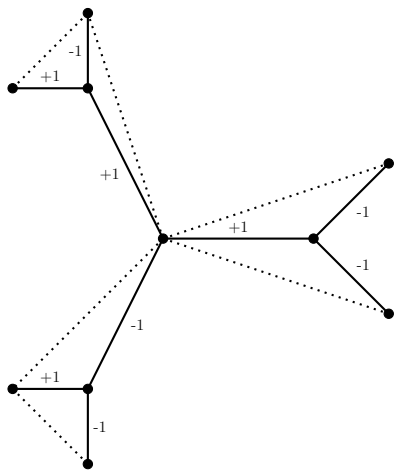}
\caption{\label{figuraejemplo}The ball of radius two centered at $x$ in a realization of the random graph $(T,x)$.}
\end{figure}

\subsection{Volume growth and Canopy Trees\label{canopysection}}

The main limitation of Theorem \ref{maintheorem} is that there is no general method to verify the growth hypothesis needed on the graph.  In fact, in this section we will show that there exist recurrent, and hence Liouville, stationary random graphs with super-exponential volume growth.  The example was communicated to the authors by Elliot Paquette who attributed it to Asaf Nachmias.

Given a sequence of natural numbers $a_1,a_2,\ldots$ construct a tree as follows:
\begin{enumerate}
 \item Begin with a single ``level-$1$'' vertex joined to $a_1$ ``level-$0$'' vertices and call the resulting tree $T_1$.
 \item Join a single (new) ``level-$2$'' vertex to the ``level-$1$'' vertices of $a_2$ copies of $T_1$ to obtain the tree $T_2$.
 \item For each $n \ge 3$ join a single ``level-$n$'' vertex to the ``level-$(n-1)$'' vertices of $a_n$ copies of $T_{n-1}$ to obtain $T_n$.
\end{enumerate}

We call the unrooted tree obtained as the union of all $T_n$ the ``Canopy tree'' determined by the sequence $a_1,a_2,\ldots$.   Notice that the isomorphism group of such a tree preserves and acts transitively on the set of level-$n$ vertices for each $n$.

In the case where the sequence $a_n$ is constant and equal to $2$, we obtain the Canopy tree as defined by Aizenman and Warzen in \cite{aizenman-warzel2006}.  One obtains a stationary random graph by rooting the graph randomly at a level-$0$ vertex with probability $1/4$, and at a level-$n$ vertex with probability $3/2^{n+2}$ for each $n \ge 1$.  This is a nice example of a recurrent graph, in particular Liouville, with exponential volume growth. The ball of radius $2n$ contains at least $2^n$ and no more than $3^{2n+1}$ vertices.

On the other hand, if all $a_n$ are equal to $1$ one obtains a single vertex at each level and there is no way of choosing a random root in a such a way that the resulting random graph is stationary.  The following lemma shows that in general one may construct a stationary graph from a Canopy tree if the sequence $a_n$ does not contain too many ones.

\begin{lemma}
The Canopy tree determined by a sequence $a_1,a_2,\ldots$ admits a random root such that the resulting graph is stationary if and only if
\[\sum \frac{1}{a_1a_2\cdots a_n} \l +\infty.\]
\end{lemma}
\begin{proof}
Consider the Markov chain on $0,1,\ldots$ with probability of going from $0$ to $1$ equal to $1$, probability $1/(a_k+1)$ of going from $k$ to $k+1$ for each $k \ge 1$, and probability $a_k/(a_k+1)$ of going from $k$ to $k-1$ for each $k \ge 1$.   There is a random root on the given Canopy tree such that the resulting random graph is stationary if and only if there is a stationary probability for the aforementioned Markov chain.

If the above series converges, then defining $p_0 = 1/S$ and $p_k = (a_k+1)/(Sa_1\cdots a_{k})$ for each $k \ge 1$, where
\[S = 1 + (a_1+1)/a_1 + (a_2+1)/(a_1a_2) + (a_3 +1)/(a_1a_2a_3) + \cdots,\]
one obtains a stationary probability for the Markov chain.

For the converse direction, assume that there is a stationary probability for the chain.  Then the expected value $m_k$ of the hitting time at $0$ for the chain started at $k$ is finite and non-negative for all $k \ge 1$. The $m_k$ satisfy the recurrence relation
\[m_k = a_k m_{k-1}/(a_k+1) + m_{k+1}/(a_k+1) + 1,\]
or equivalently 
\[m_{k+1} - m_k = a_k(m_k - m_{k-1}) - (a_k+1).\]
Setting $\Delta_k = m_{k+1}-m_k$, one obtains, using the general solution for a first order linear recurrence, that
\[\Delta_n = \left(\prod\limits_{k = 1}^{n-1}a_k\right)\left(\Delta_1 - \sum\limits_{k = 1}^{n-1} (a_k+1)/(a_1\cdots a_k)\right).\]
Since $m_k \ge 0$ for all $k$, one must have that the finite sum in the formula for $\Delta_k$ is bounded by $\Delta_1$ (which must be positive) for all $n$.  In particular, one obtains $\sum 1/(a_1\cdots a_n) \l +\infty$ as claimed.
\end{proof}

The Canopy tree determined by the sequence $a_n = n$ satisfies the hypothesis of the above lemma, and therefore can be turned in to a stationary random graph by adding the appropriate random root.  One can verify that the ball of radius $2n$ centered at any vertex of the tree contains at least $n!$ vertices and therefore the graph has super-exponential volume growth. On the other hand, the graph is recurrent and therefore Liouville.  

In the previous example the degree of the root is $k+1$ with probability $(k+1)e^{-1}/k!$ for each $k$, and has finite expectation. This example is also interesting in that the degree of the root determines the isomorphism class of the graph completely. In particular, the degree of the root and the number of elements at distance $r$ are highly dependent random variables even for large $r$. We will now show that in a unimodular graph where the degree of the root is well behaved and the number of elements of the sphere of radius $r$ is ``reasonably independent'' from the degree of the root, one can prove that there is finite exponential growth.

\begin{lemma}\label{growthlemma}
Let $(X,x)$ be a unimodular random graph such that $\E(\deg(x)) \l +\infty$ and there exists a constant $C$ such that
\begin{equation}\label{independence}\E\left(\deg(x)|S_r(x)|\right)\leq C\,\E(|S_r(x)|)\end{equation}
for all $r$, where $|S_r(y)|$ denotes the number of elements at distance $r$ from the vertex $y$ of $X$. Then $v(X,x)<+\infty$.
\end{lemma}
\begin{proof}
Let $B_r(y)$ be the graph ball of radius $r$ centered at $y\in V(X)$ and $S_r(y)$ be the respective sphere. By the triangle inequality
$$\left|B_{r+1}(x)\right|\leq \left|B_r(x)\right|+\sum_{y\in S_r(x)}\deg(y)=\left|B_r(x)\right|+\sum_{y\in V(X)}\mathds{1}_{d(x,y)=r}\deg(y).$$
Consider the function
$$F(X,x,y)=\mathds{1}_{d(x,y)=r}\deg(y).$$
Then
\begin{align*}
\E\left[\sum_{y\in V(X)}\mathds{1}_{d(x,y)=r}\deg(y)\right]&=\E\left[\sum_{y\in V(X)}F(X,x,y)\right]=\E\left[\sum_{y\in V(X)}F(X,y,x)\right]\\
&=\E\left[\sum_{y\in V(X)}\mathds{1}_{d(x,y)=r}\deg(x)\right]=\E\left[\deg(x)|S_r(x)|\right]\\
&\leq C\,\E|S_r(x)|,
\end{align*}
by our assumption. Therefore
$$\E\left|B_{r+1}(x)\right|\leq \E\left|B_{r}(x)\right|+C\E|S_r(x)|\leq (1+C)\E|B_r(x)|.$$
This implies that $v(X,x)\leq 1+C$. Recall that $(X,x)$ is stationary with respect to the probability measure $\mathbb{Q}$ defined in (\ref{bias}). Notice that condition (\ref{independence}) then implies that $v_{\mathbb{Q}}(X,x)$ (i.e. the volume growth relative to the probability measure $\mathbb{Q}$) is also finite.
\end{proof}

\subsection{Augmented Galton-Watson Tree}

We will now illustrate how Theorem \ref{maintheorem} implies a known result about harmonic functions on Galton-Watson trees.

Consider two independent Galton-Watson trees $T_1$ and $T_2$ with the same offspring distribution $\lbrace p_k: k \ge 0\rbrace$. That is, $p_k$ is the probability that a vertex has $k$ children. We assume that $p_0 = 0$ and the offspring distribution has finite mean and variance.  The Augmented Galton-Watson tree is constructed by joining the roots of $T_1$ and $T_2$ with a single edge and rooting the resulting graph at the root of $T_1$.

It has been shown in \cite{lyons-pemantle-peres1995} that under the above conditions the Augmented Galton-Watson is a stationary random graph and that the simple random walk on it escapes with positive speed given by
\[\ell = \sum\limits_{k}p_k (k-1)/(k+1).\]

Since the offspring distribution has finite mean and variance, the resulting random graph has finite exponential volume growth. See for example \cite[Chapter 12]{lyons-peres2014}.  Hence one may apply Theorem \ref{entropytheorem} to obtain that almost surely the Augmented Galton-Watson tree admits an infinite dimensional space of bounded harmonic functions.

\subsection{Hyperbolic $\kappa$-Markovian triangulations}

In a recent work \cite{curien2014} N.\,Curien has introduced a one parameter family of random infinite triangulations of the plane which generalize the Uniform Infinite Planar Triangulation (UIPT) \cite{angel-schramm2003}. These are called $\kappa$-Markovian planar triangulations where the parameter $\kappa\in (0,\kappa_0]$. The critical parameter $\kappa_0=2/27$ corresponds to the UIPT, while for $\kappa<\kappa_0$ the triangulations are hyperbolic in flavor.

It is shown in \cite{curien2014}, that in the hyperbolic regime ($\kappa<\kappa_0$) these triangulations are almost surely non-Liouville, have anchored expansion and positive linear drift. The proof of the non-Liouville property relies on the planarity of the triangulations and the fact that almost surely they do not possess the intersection property. In this section, we apply Theorem \ref{maintheorem} to provide an alternative proof of the non-Liouville property, and in fact that (when $\kappa<\kappa_0$) the $\kappa$-Markovian triangulation almost surely admits an infinite dimensional space of bounded harmonic functions.

We fix $\kappa\in (0,\kappa_0)$ for the rest of this section. The $\kappa$-Markovian infinite planar triangulation $T$ is a random rooted type II triangulation of the plane. We refer the reader to \cite[Section 1.2]{angel-schramm2003} for the precise definitions. It is defined by the following property: there exists a sequence $\{C_p\}_{p\geq 2}$ of positive numbers, which depends on $\kappa$, such that if $\tau$ is a finite rooted triangulation of the $p$-gon, then 
$$\P(\tau\subset T)=C_p\kappa^{|\tau|},$$
where $|\tau|$ is the number of vertices of $\tau$. Here $\tau\subset T$ means that $T$ is obtained from $\tau$ with coinciding roots, by filling its hole with a necessary unique infinite triangulation of the $p$-gon. By \cite[Section 3.1]{curien2014} $T$ is stationary and ergodic.

\begin{theorem}\label{kmarkovian} Let $\kappa\in (0,\kappa_0)$. The $\kappa$-Markovian triangulation $T$ almost surely admits an infinite dimensional space of bounded harmonic functions.
\end{theorem}

For any $r\geq 1$, let $B_r(o)$ denote the sub-triangulation of $T$ consisting of all the triangles which are incident to a vertex at distance less than or equal to $r-1$ from the root. Notice that since $T$ is one-ended, the complement of $B_r(o)$ has only one infinite connected component. Let $\overline{B}_r(o)$ be the hull of $B_r(o)$ obtained by filling-in all the finite components of its complement. By \cite[Theorem 2]{curien2014}, the exponential rate growth of $\overline{B}_r(o)$ is known: there exists a constant $\lambda>1$ and a random variable $V\in (0,+\infty)$, which depend only on $\kappa$, such that 
$$\lambda^{-r}\left|\overline{B}_r(o)\right| \mathrel{\mathop{\longrightarrow}^{\mathrm{a.s.}}_{r\to+\infty}} V.$$
In order to apply Theorem \ref{maintheorem}, we need to control the expected number of elements of the balls.

\begin{lemma}\label{crecimientocurien}
There exists a constant $C$, which depends only on $\kappa$, such that 
$$\E\left(\left|B_r(o)\right|\right)\leq C^r$$
for all $r\in \N$.
\end{lemma}

The proof is based on an algorithmic device, called the peeling process, that allows to construct $T$ as a sequence of growing finite triangulations $\{T_n\}_{n\geq 0}$, see \cite{angel2003}. The process starts by declaring $T_0$ to be one of the triangles that are incident to the root of $T$. At each step, $T_n$ is a finite triangulation whose boundary $\partial T_n$ consists of a simple closed curve. Suppose $T_n$ is constructed, and enumerate the boundary vertices $\partial T_n=\{x_1,\ldots,x_p\}$, where $p=|\partial T_n|$ is the perimeter of $T_n$. 

There is a triangle in $T\setminus T_n$ incident to the edge $\{x_1,x_p\}$. If we call the third vertex of this triangle $y$, there are two possibilities for the location of $y$: either $y$ is a new vertex, or $y=x_i$ for some $i\in\{2,\ldots,p-1\}$.  The probabilities of these events are given by
$$\P\left(y\notin\partial T_n\right)=\frac{\kappa C_{p+1}}{C_p},\ \P\left(y=x_i\right)=\frac{C_{p-i+1}Z_{i}}{C_p}.$$
Here $Z_i$ is the partition function associated to the Boltzmann distribution on triangulations of the $i$-gon: if we denote by $\mathcal{T}_i$ the set of all finite triangulations of the $i$-gon, then
$$Z_i=\sum_{\tau\in\mathcal{T}_i}\kappa^{|\tau|-i},$$
and for any $\tau\in \mathcal{T}_i$, the probability of $\tau$ is given by $\kappa^{|\tau|-i}/Z_i$. An explicit formula for $Z_i$ is known, see for example \cite[Section 3.3]{ray2014}, but we will not need it here. To complete the construction of $T_{n+1}$, one fills the hole created in the case when $y=x_i$ with a independent Boltzmann triangulation $\tau$ of the $i$-gon.

This process depends on the choice of an edge to peel at each step. One way to chose the edge to peel is given by the ``peeling by layers'' algorithm: at step $n$, peel the right-most edge of $\partial T_n$ which belongs to the triangle just revealed. Using this algorithm, every vertex of $\partial T_n$ will be eventually in the interior of $T_m$ for some big enough $m\geq n$, and so $T=\bigcup_{n\geq 0}T_n$.

We are interested in the Markov chain
$$(P_n,V_n)=\left(|\partial T_n|,|T_n|\right),\ n\geq 0.$$
We will use the notation $\Delta X_n=X_{n+1}-X_n$ for the increments of a sequence of random variables. The distribution of the increment $\Delta P_n$ is given by
$$\P\left(\Delta P_n=1|P_n=p\right)=\frac{\kappa C_{p+1}}{C_p},\text{ and } \P\left(\Delta P_n=-i|P_n=p\right)=2\frac{C_{p-i}Z_{i+1}}{C_p},$$
where the factor of $2$ is because there are two ways of attaching a triangle to $T_n$ with vertices $x_1,x_p$ and $x_{i+1}$. When $p$ is large, these probabilities converge to a limit distribution. In fact, by \cite[Lemma 4]{curien2014}, if we let $\alpha\in(2/3,1)$ be given by $2\kappa=\alpha^2(1-\alpha)$, then the following limit exists
$$\lim_{p\to+\infty}\beta^pC_p=c_\alpha\in(0,+\infty),$$
where $\beta=\kappa/\alpha$. The limit distribution is therefore given by
$$q_1=\alpha\text{ and }q_{-i}=2\beta^iZ_{i+1} \text{ for }i\geq 1.$$
Consider $(X_n)_{n\geq 0}$ a random walk on the integers, started at $2$, with independent increments following the distribution $\{q_1,q_{-i}\ i\geq 1\}$. In \cite[Lema 4.2]{ray2014} it is shown that this random walk has positive drift given by the expected value of the increments.

Conditionally on $(X_n)_{n\geq 0}$, define $(Y_n)_{n\geq 0}$ so that $\Delta Y_n$ are independent and distributed as the number of internal vertices of a Boltzmann triangulation of the $(-\Delta X_n+1)$-gon (if $\Delta X_n=1$, let $\Delta Y_n=1$ by convention). More precisely, the distribution of $\Delta Y_n$ is given by
$$\P(\Delta Y_n=k)=\begin{cases}q_1+\sum_{i\geq 1}q_{-i}\left|\mathcal{T}_{i+1}^{(1)}\right|\frac{\kappa}{Z_{i+1}},& \text{ if }k=1\\ \sum_{i\geq 1}q_{-i}\left|\mathcal{T}_{i+1}^{(k)}\right|\frac{\kappa^k}{Z_{i+1}}, & \text{ if }k\geq 2,\end{cases}$$
where $\mathcal{T}_{i+1}^{(k)}$ denotes the set of all triangulations of the $(i+1)$-gon with $k$ internal vertices. The exact value of $\left|\mathcal{T}_{i+1}\right|$ is known, see for example \cite[Theroem 2.1]{angel-schramm2003}, but we will not use it here.

In the proof of Lemma \ref{crecimientocurien} we will need the following fact which is a consequence of moderate deviations estimates.  From now on if $f$ and $g$ are non-negative real valued functions defined on a set $A$, we will write $f\lesssim g$ if there exists a constant $C$, such that $f(a)\leq C g(a)$ for all $a\in A$.

\begin{lemma}\label{moderate}
Let $(\xi_n)_{n\geq 1}$ be a sequence of independent identically distributed random variables, such that $\E\left(e^{t|\xi_1|}\right)<\infty$ for some $t>0$. Denote by $\mu=\E(\xi_1)$, and let $\tau\in\N$ be any random variable. Then, for all $\epsilon>0$, there exists a constant $C$ such that
$$\E\left|\sum_{i=1}^\tau \xi_i-\tau\mu\right|\leq C\,\E(\tau^{1/2+\epsilon}).$$
\end{lemma}
\begin{proof}
From \cite[Lemma 1.12]{legall2005}, for any $\epsilon>0$ we have $\P\left(A_n\right)\leq C_1e^{-n^{1/2+\epsilon}} \text{ for all }n\geq 1$, where
$$A_n=\left\{\left|\sum_{i=1}^n \xi_i-n\mu\right|>n^{1/2+\epsilon}\right\}.$$
Here, the multiplicative constant depends on $\epsilon$. Consider $\eta(\omega)=\max\{n:\omega\in A_n\}$. By Borel-Cantelli's lemma, $\eta$ is finite almost surely, and in fact
\begin{equation}\label{cotacurien1}
\P\left(\eta\geq k\right)=\P\left(\bigcup_{n\geq k}A_n\right)\lesssim \sum_{n\geq k}e^{-n^{1/2+\epsilon}}\lesssim k^{1/2+\epsilon}e^{-k^{1/2+\epsilon}},
\end{equation}
We first compute the expectation on the event $\{\tau=n\}$:
\begin{align*}
\E\left(\left|\sum_{i=1}^\tau \xi_i-\tau\mu\right|\mathds{1}_{\{\tau=n\}}\right)&=\sum_{k\geq 0}\E\left(\left|\sum_{i=1}^n \xi_i-n\mu\right|\mathds{1}_{\{\tau=n\}}\mathds{1}_{\{\eta=k\}}\right)\\
&\leq n^{1/2+\epsilon}\P\left(\tau=n\right)+\sum_{k\geq n}\E\left(\left|\sum_{i=1}^n \xi_i-n\mu\right|\mathds{1}_{\{\tau=n\}}\mathds{1}_{\{\eta=k\}}\right)
\end{align*}
By the Cauchy-Schwarz inequality, the second term in the right-hand side of the last inequality is bounded from above by
$$n^{1/2}\mathrm{Var}(\xi_1)^{1/2}\sum_{k\geq n}\P\left(\tau=n,\eta=k\right)^{1/2}\lesssim n^{1/2}\left(\sum_{k\geq n}k^{1/2+\epsilon}e^{-k^{1/2+\epsilon}}\right)^{1/2},$$
where the las inequality follows from (\ref{cotacurien1}). Bounding from above the last sum, we finally obtain the estimate
$$\E\left(\left|\sum_{i=1}^\tau \xi_i-\tau\mu\right|\mathds{1}_{\{\tau=n\}}\right)\leq n^{1/2+\epsilon}\P\left(\tau=n\right) + \gamma_n,$$
where $\gamma_n$ is a summable sequence. Summing over all the possible values for $\tau$ we get the desired inequality.
\end{proof}

\begin{proof}[Proof of Lemma \ref{crecimientocurien}]
For $r\geq 1$, let $\tau_r$ be the first time when all the vertices of $\partial T_n$ are at distance at least $r$ from the root. Then $\overline{B}_r(o)=T_{\tau_r}$, and in particular $V_{\tau_r}=\left|\overline{B}_r(o)\right|$. Notice that the increments $\Delta P_n$ are bounded from above by $1$, so for all $r$ we have $P_{\tau_r}\leq \tau_r$. First we need the following estimate which is proved in \cite[Lemma 4.2]{angel2003}:

\begin{center}
$\exists\ a,b>0$ such that $\P\left(\Delta\tau_r\geq k|P_{\tau_r}=p\right)\leq e^{-bk}$ for all $p$ and $k> ap$.
\end{center}

\noindent From this it follows that for all $p$:
\begin{align*}
\E\left(\Delta\tau_r|P_{\tau_r}=p\right)&\leq \sum_{k=0}^{ap}k\P\left(\Delta\tau_r=k|P_{\tau_r}=p\right)+\sum_{k>ap}ke^{-bk}\leq ap+\sum_{k>ap}ke^{-bk},
\end{align*}
which implies 
$$\E\left(\Delta\tau_r\right)\lesssim \E\left(P_{\tau_r}\right)\lesssim \E\left(\tau_r\right).$$
We have shown that there exists a constant $C$ so that $\E\left(\tau_r\right)\leq C^r$ for all $r\geq 1$.

The key point in what follows is that $(P_n,V_n)$ is equal in distribution to $(X_n,Y_n)$ conditioned on the event $\{X_i\geq 2,\forall i\geq 0\}$ which has positive probability, see \cite[Section 2]{curien2014}. Denote by $\gamma>0$ the probability of this event, then $\E\left(V_{\tau_r}\right)\leq \gamma^{-1}\E(Y_{\tau_r})$.

We first show that the distribution of the increments $\Delta Y_n$ have an exponential tail. Let $k\geq 2$, then
$$\P\left(\Delta Y_n=k\right)=\kappa^k\sum_{i\geq 1}\beta^i\left|\mathcal{T}_{i+1}^{(k)}\right|.$$
Notice that $\beta$ is a monotone function of $\kappa$, so if we take $\kappa<\kappa'<\kappa_0$, the corresponding $\beta'$ satisfies $\beta<\beta'<\beta_0=1/9$. Therefore
$$\P\left(\Delta Y_n=k\right)=\kappa^k\sum_{i\geq 1}\beta^i\left|\mathcal{T}_{i+1}^{(k)}\right|\leq \left(\frac{\kappa}{\kappa'}\right)^k\left[(\kappa')^k\sum_{i\geq 1}(\beta')^i\left|\mathcal{T}_{i+1}^{(k)}\right|\right]\leq \left(\frac{\kappa}{\kappa'}\right)^k.$$
In particular, we are in the hypotheses of Lemma \ref{moderate}. Let $\mu=\E\left(\Delta Y_1\right)$, then for any $\epsilon>0$, there exists a constant $C$ such that
$$\E\left(Y_{\tau_r}\right)\leq \mu\,\E\left(\tau_r\right)+C\,\E\left(\tau^{1/2+\epsilon}\right).$$
Fix $\epsilon\in (0,1/2)$, then by Jensen's inequality we obtain $\E\left(Y_{\tau_r}\right)\lesssim\E(\tau_r)$. This finishes the proof of the lemma.
\end{proof}

\subsection{Poisson Delaunay random graphs}

In this section we denote by $M$ either the $d$-dimensional Euclidean space $\R^d$, or the $d$-dimensional hyperbolic space $\H^d$. In the latter case, we use the Poincar\'e ball model where in polar coordinates the metric is given by
$$\d s^2=\d r^2+\sinh(r)^2\d \theta^2,$$
where $\d \theta^2$ is the standard metric on the sphere $\mathbb{S}^{d-1}$. In both cases we write $\d x$ for the volume element on $M$.

\begin{figure}
\begin{center}
\vspace{0.5cm}
\includegraphics[scale = 0.25]{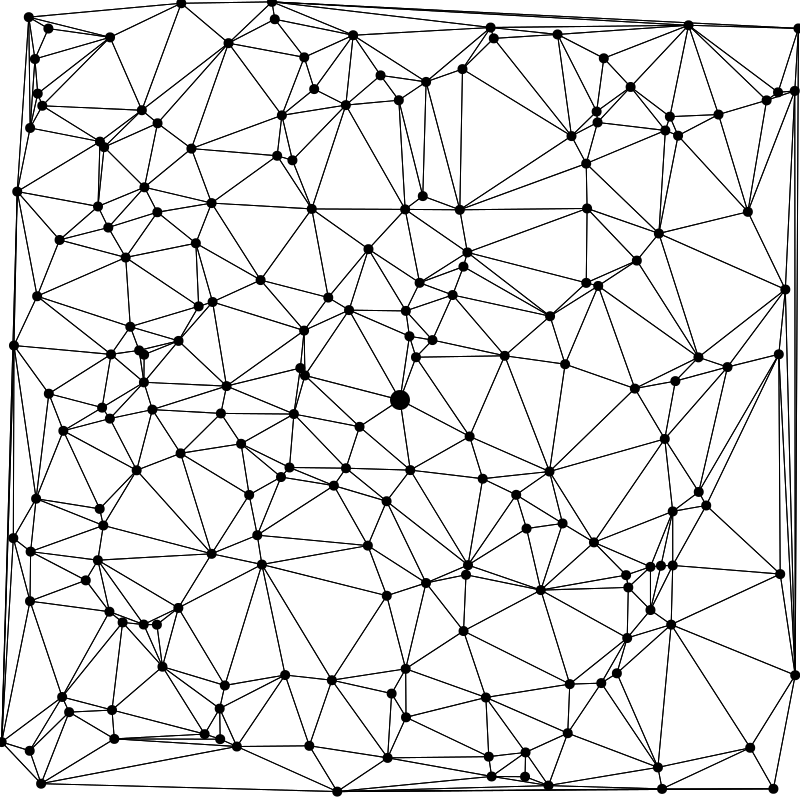}
\end{center}
\caption{\label{poissondelaunayfigure}An approximate Poisson-Delaunay triangulation} This is the Delaunay triangulation associated to a set of independent uniform points in $[-1,1]^2 \subset \R^2$ and rooted at the origin.  As the number of points increases the resulting random graph approximates the Poisson-Delaunay random graph in distribution.
\end{figure}

Let $\Pi$ be a homogeneous Poisson point process of intensity one on $M$. That is, $\Pi$ is a random discrete set of points on $M$ with the following properties: 
\begin{enumerate}
\item the number of points in any Borel set $A$ is a Poisson random variable whose expected value is the volume of $A$;
\item for any two disjoint Borel sets $A$ and $B$, the corresponding Poisson random variables are independent.
\end{enumerate}
We refer to \cite{kingman1993} for an introduction to point processes. We chose the intensity to be one only for simplicity, but the arguments given here go through for the general case of constant intensity with out significant changes.

We set $o = 0 \in M$, and consider the Delaunay graph associated to the discrete set 
$$\Pi_o = \Pi \cup \lbrace o\rbrace.$$
That is, we consider the random rooted graph $X$ with root $o$ and vertex set $\Pi_o$ such that two vertices $x,y \in X$ are joined by a single undirected edge if, and only if, there exists a ball with $x$ and $y$ on its boundary and whose interior contains no points of $\Pi_o$.  The resulting random graph is almost surely the dual graph of a tessellation of $M$ into simplices known as the Voronoi tessellation. See Figure \ref{poissondelaunayfigure} for a realization in $\R^2$.

The main goal of this section is to prove the following

\begin{theorem}\leavevmode \label{Delaunay}
\begin{enumerate}
\item The Euclidean Poisson-Delaunay random graph is almost surely Liouville for all $d$.
\item The Hyperbolic Poisson-Delaunay random graph almost surely admits an infinite dimensional space of bounded harmonic functions for $d=2$.
\end{enumerate}
\end{theorem}

We will first establish that the growth assumption (\ref{assumption1}) on Lemma \ref{speedlemma} is satisfied. The main tool we will use is Slivnyak's formula, see \cite[Proposition 4.1.1.]{moller1994}, which we will now restate for the reader's convenience.
\begin{lemma}[Slivnyak's formula]
Let $\Pi$ be a Poisson process on $M$ with intensity $1$.  For every measurable function $f:M^{*n} \times \mathcal{D} \to [0,+\infty)$, where $M^{*n}$ is the space of $n$-element subsets of $M$ and $\mathcal{D}$ is the space of discrete subsets of $M$, one has
\begin{dmath*}\E\left(\sum\limits_{\lbrace x_1,\ldots,x_n\rbrace \subset \Pi}f\Big(\lbrace x_1,\ldots,x_n\rbrace,\Pi\Big)\right) = {\frac{1}{n!} \int\limits_{M^n} \E\Big[f\Big(\lbrace y_1,\ldots,y_n\rbrace,\Pi\cup \lbrace y_1,\ldots,y_n\rbrace\Big)\Big]\d(y_1,\ldots,y_n)}.\end{dmath*}
\end{lemma}

We will now prove that in the Euclidean case the Poisson-Delaunay graph grows sub-exponentially.

\begin{lemma}\label{eucdelaunay}
Let $d\geq 1$, and $X$ be the Poisson-Delaunay graph rooted at $o = 0 \in \R^d$. For each $r\in\N$, we denote by $B_r(o)$ the ball of radius $r$ centered at $o$ in $X$. Then:
\begin{enumerate}
\item Almost surely $|B_r(o)|=\mathrm{O}(r^d\log^dr)$ when $r\to +\infty$, and
\item $\E(\deg(o))<+\infty$ and $\E_{\mathbb{Q}}{|B_r(o)|}=\mathrm{O}(r^d\log^dr)$ when $r\to+\infty$.
\end{enumerate}
Recall that $\mathbb{Q}$ is the probability measure defined in (\ref{bias}). In particular, $X$ has polynomial volume growth and $v_{\mathbb{Q}}(X,o)=0$.
\end{lemma}
\begin{proof}
We will prove \emph{2}. The first assertion follows from the proof by applying Borel-Cantelli's Lemma. Let $\Pi$ be the Poisson point process in $\R^d$ with intensity $1$, and define $L_\Pi(x)$ to be the Euclidean distance between $x$ and its farthest neighbor in the Delaunay graph of $\Pi \cup \lbrace x\rbrace$.

The proof relies on the following exponential bound for the tail of $L_\Pi(x)$: there is a positive constant $c$ such that
\begin{equation}\label{colaexponencial}\P\left( L_\Pi(x) \g s\right) \lesssim e^{-cs},\end{equation}
for all $s \g 0$. See for example \cite{moller1994,boots-okabe-sugihara1992}.

Let us denote by $B^M_r(o)$ the Euclidean ball of radius $r$ centered at $o$. Let $\{s_r\}_{r\geq 0}$ be a monotone sequence of non-negative numbers, with $r_0=0$, to be chosen later and define $S_r = s_1 + \cdots s_{r-1}$. Applying (\ref{colaexponencial}), we will bound from above the probability that there exists an edge in the Delaunay graph $X$, of Euclidean length at least $s_r$, starting at a point of $B_{S_r}^M(o)$.

Recall that $\mathcal{D}$ denotes the space of discrete subsets of $\mathbb{R}^d$, and consider the function $f:\mathcal{D}\times\mathbb{R}^d\to \R$ given by
$$f(Z,z)=\mathds{1}_{Z \cap B_{S_r}^M(o)}(z)\mathds{1}_{\left\{L_{Z}(z) \g s_r\right\}},\ Z\in\mathcal{D}\text{ and }z\in\R^d.$$
By Slyvniak's formula, we have
\begin{align*}
\P\left(\exists\ x \in \Pi_o \cap B_{S_r}^M(o): L_{\Pi_o}(x) \g s_r\right) &\leq \mathbb{E}\left[\sum_{x\in \Pi_o}f(\Pi_o,x)\right]\\
&=\int_{\mathbb{R}^d}\mathbb{E}\left[f(\Pi_o\cup\{y\},y)\right]\d y\\
&=\int_{B_{S_r}^M(o)}\P\left(L_{\Pi_o}(y) \g s_r\right)\d y\\
&\le \int\limits_{B_{S_r}^M(o)}\P\left(L_{\Pi}(y) \g s_r\right)\d y
\\ &\lesssim S_r^d e^{-cs_r} ,
\end{align*}
Notice that the same upper bound holds if we add more points to $\Pi$ instead of adding just the point $o$. Suppose now that $\{s_r\}$ is chosen so that 
$$\sum_{k=1}^\infty\sum_{r\geq k}S_r^de^{-cs_r}<\infty.$$
Consider the sequence of events
$$A_r=\left\{Z\in\mathcal{D}:\exists\ z \in Z \cap B_{S_r}^M(o)\text{ s.t. } L_{Z}(z) \g s_r\right\},$$
and define $K(Z)=\max\{r:Z\in A_r\}$. Then $\{K\geq k\}=\bigcup_{r\geq k}A_r$, and therefore
$$\E(K)\leq\sum_{k=1}^\infty\sum_{r\geq k}\mathbb{P}(A_r)\lesssim \sum_{k=1}^\infty\sum_{r\geq k} S_r^d e^{-cs_r}<\infty.$$
The event $\{K=k\}$ is
$$\Big\{Z\in\mathcal{D}:\exists\ x\in Z\cap B_{S_k}^M(o)\text{ with }L_Z(x)>s_k\text{ and }L_Z(z)\leq s_r\ \forall z\in Z\cap B_{S_r}^M(o)\ \forall r>k \Big\}.$$
Suppose that $\Pi_o$ satisfies $K(\Pi_o)=k$. This implies that for any $r \ge k$ and any point $x \in \Pi_o \setminus B_{S_{r-1}}$, the distance in the graph $X$ between $o$ and $x$ is at least $r-k$. In other words, we have
$$B_r(o)\subset B_{S_{r+K(\Pi_o)}}^M(o).$$
Decompose the second moment of $|B_r(o)|$ according to the values of $K$,
$$\E{|B_r(o)|^2}=\sum_{k=1}^\infty\E\Big[|B_r(o)|^2\mathds{1}_{\{K(\Pi_o)=k\}}\Big].$$
Then, we obtain the upper bound
$$\E\left|B_r(o)\right|^2\leq \sum_{k=1}^\infty\E\left[\left|\Pi_o\cap B_{S_{r+k}}^M(o)\right|^2\mathds{1}_{\{K(\Pi_o)=k\}}\right].$$
Let $g:\mathcal{D}\times \mathbb{R}^d\to\R$ be the function
$$g(Z,z)=\mathds{1}_{Z\cap B_{S_{r+k}}^M(o)}(z)\mathds{1}_{\{K(Z)=k\}}.$$
By Slyvniak's formula again, we get
\begin{align*}\mathbb{E}\Big[\left|\Pi_o\cap B_{S_{r+k}}^M(o)\right|^2\mathds{1}_{\left\{K(\Pi_o)=k\right\}}\Big]
&=\mathbb{E}\left[\sum_{x_1,x_2\in \Pi_o}g(\Pi_o,x_1)g(\Pi_o,x_2)\right]\\
&=\int_{(\mathbb{R}^d)^2}\mathbb{E}\Big[g(\Pi_o\cup\{y_1,y_2\},y_1)g(\Pi_o\cup\{y_1,y_2\},y_2)\Big]\d(y_1,y_2)\\
&=\int_{B_{S_{r+k}}^M(o)^2}\mathbb{P}\left[\mathds{1}_{\left\{K(\Pi_o\cup\{y_1,y_2\})=k\right\}}\right]\d(y_1,y_2)\\
&\lesssim S_{r+k}^{2d}\sum_{m\geq k}S_m^de^{-cs_m}
\end{align*}
The last inequality follows from
$$\mathbb{P}(K(\Pi_o\cup\{y_1,y_2\})\geq k)\lesssim \sum_{m\geq k} S_m^de^{-cs_m}.$$
Therefore, we obtain
$$\E|B_r(o)|^2\lesssim \sum_{k=1}^\infty S_{r+k}^{2d}\sum_{m\geq k}S_m^de^{-cs_m}.$$
We set $s_r=\frac{\alpha}{c}\log r$, for some fixed $\alpha>3d+2$, so that
$$S_r=\frac{\alpha}{c}\log((r-1)!)\leq \frac{\alpha}{c}r\log r.$$
Notice that there $S_{r+k}\lesssim S_rS_k$. On the other hand, we have
$$\sum_{m\geq k}S_{m}^de^{-cs_m}\leq \frac{\alpha^d}{c^d}\sum_{m\geq k}\frac{\log^d m}{m^{\alpha-d}}\lesssim \frac{\log^d k}{k^{\alpha-d-1}}.$$
From this, we get
$$\sum_{k=1}^\infty S^{2d}_{r+k}\sum_{m\geq k}S_m^de^{-cs_m}\lesssim S_r^{2d}\sum_{k\geq 1}\frac{\log^{3d}(k)}{k^{\alpha-(3d+1)}}=\mathrm{O}(S_r^{2d}).$$
This gives the upper bound $\E|B_r(o)|^2=\mathrm{O}(S_r^{2d})$. By definition, we have
$$\E_{\mathbb{Q}}|B_r(o)|=\E\left[\deg(o)|B_r(o)|\right]\leq \left(\E\deg(o)^2\right)^{1/2}\left(\E|B_r(o)|^2\right)^{1/2}=O(S_r^{d}),$$
where the last step follows from Cauchy-Schwarz inequality. This concludes the proof of \emph{2}.
\end{proof}

In order to estimate the growth of the Hyperbolic Poisson-Delaunay graph we use a result proved in \cite{benjamini-paquette-pfeffer2014}.

\begin{lemma}\label{hypdelaunay} Let $X$ be the Poisson-Delaunay graph rooted at $o = 0 \in \H^d$. For each $r\in\N$, we denote by $B_r(o)$ the ball of radius $r$ centered at $o$ in $X$. Then $\E(\deg(o))<+\infty$ and there exists a constant $L$ such that
$$\E_{\mathbb{Q}}|B_r(o)|=O(e^{L r}) \text{ when }r\to+\infty.$$
In particular,
$$v_{\mathbb{Q}}(X,o)=\liminf_{r\to\infty}\frac{1}{r}\,\E_{\mathbb{Q}}\log |B_r(o)|<+\infty.$$
\end{lemma}
\begin{proof}
For $r>0$, let $B_r^M(o)$ be the hyperbolic ball of radius $r$ centered at $o=0\in \H^d$. By \cite[Proposition 4.1]{benjamini-paquette-pfeffer2014}, the following estimate holds: there are constants $\delta>0$ and $L_0>0$ such that for any $L\geq L_0$ 
$$\P\left(B_r(o)\not\subset B_{L r}^M(o)\right)\leq e^{-cr},$$
where $c=e^{\delta L}$. We fix $L\geq L_0$ such that $2L<c$. Consider $K(\Pi_o)=\max\{r:B_r(o)\not\subset B_{L r}^M(o)\}$. Then, using the previous estimate, we obtain
$$\P(K\geq k)\leq \sum_{r\geq k}e^{-cr}\lesssim e^{-ck},$$
We decompose the expectation according to the values of $K$,
\begin{align*}
\E|B_r(o)|^2&=\sum_{k=1}^\infty\E\left[|B_r(o)|^2\mathds{1}_{\{K=k\}}\right]\\
&\leq\sum_{k=1}^{r-1}\E\left[\left|\Pi_o\cap B_{L r}^M(o)\right|^2\mathds{1}_{\{K=k\}}\right]+\sum_{k=r}^\infty\E\left[|B_{k+1}(o)|^2\mathds{1}_{\{K=k\}}\right]\\
&\leq \sum_{k=1}^{r-1}\E\left[\left|\Pi_o\cap B_{L r}^M(o)\right|^2\mathds{1}_{\{K=k\}}\right]+\sum_{k=r}^\infty\E\left[\left|\Pi_o\cap B_{L (k+1)}^M(o)\right|^2\mathds{1}_{\{K=k\}}\right]
\end{align*}
We first bound from above the second term of the right hand side. As before, let $g:\mathcal{D}\times \H^d\to\R$ be the function
$$g(Z,z)=\mathds{1}_{Z\cap B_{L(k+1)}^M(o)}(z)\mathds{1}_{\{K(Z)=k\}}.$$
By Slyvniak's formula
\begin{align*}
\E\left[\left|\Pi_o\cap B_{L (k+1)}^M(o)\right|^2\mathds{1}_{\{K=k\}}\right]&=\int_{B_{L(k+1)}^M(o)^2}\P(K(\Pi_o\cup\{y_1,y_2\})=k)\d(y_1,y_2)\\
&\lesssim e^{2L(k+1)}e^{-ck}\lesssim e^{(2L-c)(k+1)}
\end{align*}
This implies that
$$\sum_{k=r}^\infty\E\left[\left|\Pi_o\cap B_{L (k+1)}^M(o)\right|^2\mathds{1}_{\{K=k\}}\right]\lesssim \sum_{k\geq r}e^{(2L-c)(k+1)}=C_r,$$
and $C_r\to 0$ when $r\to\infty$.

For the first term, we have
\begin{align*}
\sum_{k=1}^{r-1}\E\left[\left|\Pi_o\cap B_{L r}^M(o)\right|^2\mathds{1}_{\{K=k\}}\right]&\leq \E\left[\left|\Pi_o\cap B_{L r}^M(o)\right|^2\right]\lesssim e^{2L r}.
\end{align*}
In summary, we obtained
$$\E|B_r(o)|^2\lesssim e^{2L r}+C_r=O(e^{2Lr}).$$
The proof concludes as in the previous lemma by applying the Cauchy-Schwarz inequality.
\end{proof}

We now establish unimodularity of the Poisson-Delaunay graphs in both the Euclidean and Hyperbolic case using Slivnyak's formula. The proof also applies to the more general case when $M$ is a symmetric space. A different proof of the same result is given in \cite{benjamini-paquette-pfeffer2014}.

\begin{lemma}
The Poisson-Delaunay graph rooted at $o = 0 \in M$ is unimodular.
\end{lemma}
\begin{proof}
Given a discrete subset $D$ of $M$ and two points $x,y \in D$, let $\pi(D,x,y)$ be the Delaunay graph associated to $D$ with two ordered root vertices corresponding to $x$ and $y$. The codomain of $\pi$ is the space of isomorphism classes of graphs with two roots.

Recall that $\Pi$ is a Poisson point process with intensity $1$.  The aim is to show that the Delaunay graph associated to $\Pi_o$ rooted at $o$ is unimodular.  From Slivnyak's formula, we obtain for any measurable function on the space of graphs with two roots:
\[\E\left(\sum\limits_{x \in \Pi_o}F(\pi(\Pi_o,o,x))\right) = \int_{M} \E\left(F(\pi(\Pi\cup \lbrace o,y\rbrace,o,y))\right) \d\mu(y).\]

For each $y \in M$, the expected value in the integral on the right-hand side can be written as
\[\E\left(F(\pi(\Pi\cup \lbrace o,y\rbrace,o,y))\right) = \E\left(F(\pi(\Pi'\cup \lbrace o,y\rbrace,o,y))\right)\]
where $\Pi'$ is obtained from $\Pi$ by symmetry with respect to the midpoint of the geodesic segment $[o,y]$.  The equality follows because the distribution of $\Pi$ is invariant under isometries of $M$.

Next notice that $\pi(\Pi' \cup \lbrace o,y\rbrace,o,y) = \pi(\Pi\cup \lbrace o,y\rbrace,y,o)$, that is, the two graphs are isomorphic with an isomorphism which preserves the ordered basepoints. Hence, applying Slivnyak's formula again, we obtain
\[\E\left(\sum\limits_{x \in \Pi}F(\pi(\Pi,o,x))\right) = \E\left(\sum\limits_{x \in \Pi}F(\pi(\Pi,x,o))\right),\]
so that the Poisson-Delaunay graph is unimodular as claimed.
\end{proof}

The first part of Theorem \ref{Delaunay} follows directly from Lemma \ref{eucdelaunay}, Theorem \ref{maintheorem} and Lemma \ref{speedlemma}, since in that case $v_{\mathbb{Q}}(X,o)=0$. In the second part, by \cite[Theorem 1.\,1 and Theorem 1.\,4]{benjamini-paquette-pfeffer2014}, the random walk on the 2-dimensional Hyperbolic Delaunay graph has positive linear drift, so we can conclude as before using Lemma \ref{hypdelaunay}.  As far as the authors are aware there there are no results on the speed of the random walk on Hyperbolic Poisson-Delaunay graphs of dimension $d > 2$ in the literature.

\input{main.bbl}

\section*{Acknowledgments}

The authors would like to thank Itai Benjamini and Nicolas Curien for their help and encouragement during the course of this work, and Benjamini, Paquette, and Pfeffer for suggesting the application to the Hyperbolic Poisson Delaunay tesselation and sharing their ongoing work on this random graph with us (see \cite{benjamini-paquette-pfeffer2014}).
\end{document}

%% file: main.bbl
\def\cprime{$'$}

%% file: main.bbl
\begin{thebibliography}{{Gra}11}

\bibitem[Ang03]{angel2003}
O.~Angel.
\newblock Growth and percolation on the uniform infinite planar triangulation.
\newblock {\em Geom. Funct. Anal.}, 13(5):935--974, 2003.

\bibitem[AS03]{angel-schramm2003}
Omer Angel and Oded Schramm.
\newblock Uniform infinite planar triangulations.
\newblock {\em Comm. Math. Phys.}, 241(2-3):191--213, 2003.

\bibitem[AW06]{aizenman-warzel2006}
Michael Aizenman and Simone Warzel.
\newblock The canopy graph and level statistics for random operators on trees.
\newblock {\em Math. Phys. Anal. Geom.}, 9(4):291--333 (2007), 2006.

\bibitem[BC12]{benjamini-curien2012}
Itai Benjamini and Nicolas Curien.
\newblock Ergodic theory on stationary random graphs.
\newblock {\em Electron. J. Probab.}, 17:no. 93, 1--20, 2012.

\bibitem[Ben11]{benjamini2011}
Itai Benjamini.
\newblock {\em Coarse geometry and randomness}, volume 2100 of {\em Lecture
  Notes in Mathematics}.
\newblock Springer, Heidelberg, 2011.
\newblock Lecture notes from the 41st Probability Summer School held in
  Saint-Flour, 2011, {\'E}cole d'{\'E}t{\'e} de Probabilit{\'e}s de
  Saint-Flour. [Saint-Flour Probability Summer School].

\bibitem[Bla55]{blackwell1955}
David Blackwell.
\newblock On transient {M}arkov processes with a countable number of states and
  stationary transition probabilities.
\newblock {\em Ann. Math. Statist.}, 26:654--658, 1955.

\bibitem[BPP14]{benjamini-paquette-pfeffer2014}
I.~Benjamini, E.~Paquette, and J.~Pfeffer.
\newblock Anchored expansion, speed, and the hyperbolic poisson voronoi
  tessellation.
\newblock {\em ArXiv e-prints}, 2014.

\bibitem[BS96]{benjamini-schramm1996}
Itai Benjamini and Oded Schramm.
\newblock Harmonic functions on planar and almost planar graphs and manifolds,
  via circle packings.
\newblock {\em Invent. Math.}, 126(3):565--587, 1996.

\bibitem[{Cur}14]{curien2014}
N.~{Curien}.
\newblock {Planar stochastic hyperbolic infinite triangulations}.
\newblock {\em ArXiv e-prints}, January 2014.

\bibitem[Der76]{derriennic1976}
Yves Derriennic.
\newblock Lois ``z\'ero ou deux'' pour les processus de {M}arkov.
  {A}pplications aux marches al\'eatoires.
\newblock {\em Ann. Inst. H. Poincar\'e Sect. B (N.S.)}, 12(2):111--129, 1976.

\bibitem[Der85]{derriennic1985}
Y.~Derriennic.
\newblock {\em Entropie, th\'eor\`emes limite et marches al\'eatoires}.
\newblock Publications de l'Institut de Recherche Math\'ematique de Rennes.
  [Publications of the Rennes Institute of Mathematical Research]. Universit\'e
  de Rennes I Institut de Recherche Math\'ematique de Rennes, Rennes, 1985.

\bibitem[Doo01]{doob2001}
Joseph~L. Doob.
\newblock {\em Classical potential theory and its probabilistic counterpart}.
\newblock Classics in Mathematics. Springer-Verlag, Berlin, 2001.
\newblock Reprint of the 1984 edition.

\bibitem[Ers11]{erschler2011}
Anna Erschler.
\newblock Poisson-{F}urstenberg boundary of random walks on wreath products and
  free metabelian groups.
\newblock {\em Comment. Math. Helv.}, 86(1):113--143, 2011.

\bibitem[Ger88]{gerl1988}
Peter Gerl.
\newblock Random walks on graphs with a strong isoperimetric property.
\newblock {\em J. Theoret. Probab.}, 1(2):171--187, 1988.

\bibitem[{Gra}11]{gray2011}
Robert~M. {Gray}.
\newblock {\em {Entropy and information theory. 2nd ed.}}
\newblock New York, NY: Springer, 2nd ed. edition, 2011.

\bibitem[Kai92]{kaimanovich1992}
Vadim~A. Kaimanovich.
\newblock Measure-theoretic boundaries of {M}arkov chains, {$0$}-{$2$} laws and
  entropy.
\newblock In {\em Harmonic analysis and discrete potential theory ({F}rascati,
  1991)}, pages 145--180. Plenum, New York, 1992.

\bibitem[Kin93]{kingman1993}
J.~F.~C. Kingman.
\newblock {\em Poisson processes}, volume~3 of {\em Oxford Studies in
  Probability}.
\newblock The Clarendon Press, Oxford University Press, New York, 1993.
\newblock Oxford Science Publications.

\bibitem[KV83]{kaimanovich-vershik1983}
V.~A. Ka{\u\i}manovich and A.~M. Vershik.
\newblock Random walks on discrete groups: boundary and entropy.
\newblock {\em Ann. Probab.}, 11(3):457--490, 1983.

\bibitem[LG05]{legall2005}
Jean-Fran{\c{c}}ois Le~Gall.
\newblock Random trees and applications.
\newblock {\em Probab. Surv.}, 2:245--311, 2005.

\bibitem[LP14a]{lyons-peres2014b}
Russel Lyons and Yuval Peres.
\newblock In preparation., 2014.

\bibitem[LP14b]{lyons-peres2014}
Russel Lyons and Yuval Peres.
\newblock Probability on trees and networks.
\newblock \url{http://mypage.iu.edu/\string~rdlyons/}, 2014.

\bibitem[LPP95]{lyons-pemantle-peres1995}
Russell Lyons, Robin Pemantle, and Yuval Peres.
\newblock Ergodic theory on {G}alton-{W}atson trees: speed of random walk and
  dimension of harmonic measure.
\newblock {\em Ergodic Theory Dynam. Systems}, 15(3):593--619, 1995.

\bibitem[Lyo83]{lyons1983}
Terry Lyons.
\newblock A simple criterion for transience of a reversible {M}arkov chain.
\newblock {\em Ann. Probab.}, 11(2):393--402, 1983.

\bibitem[M{\o}l94]{moller1994}
Jesper M{\o}ller.
\newblock {\em Lectures on random {V}orono\u\i\ tessellations}, volume~87 of
  {\em Lecture Notes in Statistics}.
\newblock Springer-Verlag, New York, 1994.

\bibitem[OBS92]{boots-okabe-sugihara1992}
Atsuyuki Okabe, Barry Boots, and K{\=o}kichi Sugihara.
\newblock {\em Spatial tessellations: concepts and applications of
  {V}orono\u\i\ diagrams}.
\newblock Wiley Series in Probability and Mathematical Statistics: Applied
  Probability and Statistics. John Wiley \& Sons, Ltd., Chichester, 1992.
\newblock With a foreword by D. G. Kendall.

\bibitem[Pin64]{pinsker1964}
M.S. Pinsker.
\newblock {\em {Information and information stability of random variables and
  processes.}}
\newblock {Holden-Day Series in Time Series Analysis. San
  Francisco-London-Amsterdam: Holden-Day, Inc. XII. 243 p. }, 1964.

\bibitem[Ray14]{ray2014}
Gourab Ray.
\newblock Geometry and percolation on half planar triangulations.
\newblock {\em Electron. J. Probab.}, 19:no. 47, 28, 2014.

\bibitem[Vir00]{virag2000}
B.~Vir{\'a}g.
\newblock Anchored expansion and random walk.
\newblock {\em Geom. Funct. Anal.}, 10(6):1588--1605, 2000.

\end{thebibliography}
